\documentclass[twoside,11pt]{amsart}
\usepackage{amsmath,amssymb,amsbsy,amstext,amsxtra,latexsym}
\usepackage{graphics, epsfig}
\usepackage[all]{xy}
\usepackage{hyperref}

\usepackage{ccaption}
\newfixedcaption{\figcaption}{figure}

\newtheorem{theorem}{Theorem}[section]
\newtheorem{lemma}{Lemma}[section]
\newtheorem{proposition}{Proposition}[section]

\newcommand{\uc}[1]{\ensuremath \overset{#1}{\circ}}
\newcommand{\nc}[1]{\ensuremath \overset{#1}{\circ}}
\newcommand{\ntc}[2]{\ensuremath \underset{#2}{\overset{#1}{\circ}}}
\newcommand{\dc}[1]{\ensuremath \underset{#1}{\circ}}

\def\mP{{\mathbb P}}
\def\mQ{{\mathbb Q}}
\def\mC{{\mathbb C}}

\def\mZ{{\mathbb Z}}
\def\cO{{\mathcal O}}
\def\cX{{\mathcal X}}

\def\cV{{\mathcal V}}
\def\n{\noindent}
\def\b{\bigskip}
\def\m{\medskip}

\def\T{{\mathbb T}}
\def\Ext{\operatorname{Ext}}
\def\ker{\operatorname{ker}}

\def\log{\operatorname{log}}

\begin{document}

\title{A simply connected surface of general type \\
       with $p_g=0$ and $K^2=3$}

\author{Heesang Park, Jongil Park and Dongsoo Shin}

\address{Department of Mathematical Sciences, Seoul
         National University,
         San 56-1, Sillim-dong, Gwanak-gu, Seoul 151-747, Korea}

\email{hspark@math.snu.ac.kr}

\address{Department of Mathematical Sciences, Seoul
         National University,
         San 56-1, Sillim-dong, Gwanak-gu, Seoul 151-747, Korea}

\email{jipark@math.snu.ac.kr}

\address{Department of Mathematical Sciences, Seoul
         National University,
         San 56-1, Sillim-dong, Gwanak-gu, Seoul 151-747, Korea}

\email{dsshin@math.snu.ac.kr}

\date{July 31, 2007; revised at August 11, 2007}



\begin{abstract}
 Motivated by a recent result of Y. Lee and the second author~\cite{LP},
 we construct a simply connected minimal complex surface
 of general type with $p_g=0$ and $K^2 =3$ using a rational blow-down surgery
 and $\mQ$-Gorenstein smoothing theory.
 In a similar fashion, we also construct a new simply connected
 symplectic $4$-manifold with $b_2^+=1$ and $K^2 =4$.
\end{abstract}

\maketitle

\section{Introduction}
\label{sec-1}

\markboth{HEESANG PARK, JONGIL PARK AND DONGSOO SHIN}{A SIMPLY
          CONNECTED SURFACE OF GENERAL TYPE WITH $p_g=0$ AND $K^2=3$}

 One of the fundamental problems in the classification of complex surfaces
 is to find a new family of simply connected surfaces of general type with
 $p_g=0$. Although a large number of non-simply connected complex
 surfaces of general type with $p_g =0$ have been known~\cite{BHPV},
 simply connected surfaces of general type with $p_g=0$ are little
 known except Barlow surface~\cite{B}.

 Recently, the second author constructed a simply connected
 symplectic $4$-manifold with $b_2^+ =1$ and $K^2 =2$ using a
 rational blow-down surgery~\cite{P2}, and then
 Y. Lee and the second author constructed a family of
 simply connected, minimal, complex surface of general type with
 $p_g=0$ and $1 \le K^2 \le 2$ by modifying
 Park's symplectic $4$-manifold~\cite{LP}.
 After this construction,
 it has been a natural question whether one can find a new family of surfaces
 of general type with $p_g=0$ and $K^2 \ge 3$ using the same technique.

 The aim of this article is to give an affirmative answer for the
 question above. Precisely, we are able to construct a simply connected,
 minimal, complex surface of general type with $p_g=0$ and $K^2=3$ using
 a rational blow-down surgery and a $\mQ$-Gorenstein smoothing theory
 developed in~\cite{LP}. The key ingredient for the construction of $K^2=3$
 case is to find a right rational surface $Z$ which contains several disjoint
 chains of curves representing the resolution graphs of special quotient
 singularities. Once we have a right candidate $Z$ for $K^2=3$, the
 remaining argument is parallel to that of $K^2=2$ case appeared in~\cite{LP}.
 That is, we contract these chains of curves from the rational
 surface $Z$ to produce a projective surface $X$ with special quotient
 singularities.
 And then
 we prove that the singular surface $X$ has a $\mQ$-Gorenstein smoothing
 and the general fiber $X_t$ of the $\mQ$-Gorenstein smoothing is a simply
 connected minimal surface of general type with $p_g=0$ and $K^2 =3$.
 The main result of this article is the following.

\begin{theorem}
\label{thm-main}
  There exists a simply connected, minimal, complex surface
  of general type with $p_g=0$ and $K^2 =3$.
\end{theorem}

 By using different pencils and fibrations, we provide more examples
 of simply connected minimal complex surfaces of general type with
 $p_g=0$ and $K^2=3$ in \mbox{Section 6}.
 Furthermore, we also construct a new simply connected closed symplectic
 $4$-manifold with $b_2^+=1$ and $K^2 =4$ using the same technique
 as in Section 3.

\begin{theorem}
\label{thm-main2}
  There exists a simply connected  symplectic $4$-manifold with
  $b_2^+=1$ and $K^2 =4$ which is homeomorphic, but not diffeomorphic,
  to a rational surface $\mP^2\sharp 5\overline{\mP}^2$.
\end{theorem}

 It is a very intriguing question whether the symplectic $4$-manifold
 constructed in Theorem~\ref{thm-main2} above admits a complex structure.
 One way to approach this problem is to use $\mQ$-Gorenstein smoothing theory
 as above.
 But, since the cohomology $H^2(T^0_X)$ is not zero in this case,
 it is hard to determine whether there exists a $\mQ$-Gorenstein smoothing.
 Therefore we need to develop more $\mQ$-Gorenstein smoothing theory in
 order to investigate the existence of a complex structure on the symplectic
 $4$-manifold constructed in Theorem~\ref{thm-main2} above.
 We leave this question for future research.

\m

\n{\em Acknowledgements}.
 The authors would like to thank Yongnam Lee for his careful reading
 the first draft of this article.
 The authors also thank Ki-Heon Yun for his spending time with
 authors during the course of this work.
 Jongil Park was supported by SBS Foundation Grant in 2007 and he also
 holds a joint appointment in the Research Institute of Mathematics,
 Seoul National University.
 Dongsoo Shin was supported by Korea Research Foundation Grant funded
 by the Korean Government (KRF-2005-070-C00005).

\b

\section{$\mQ$-Gorenstein smoothing}
\label{sec-2}

 In this section we briefly review a theory of $\mQ$-Gorenstein smoothing
 for projective surfaces with special quotient singularities and we
 quote some basic facts developed in~\cite{LP}.

\m

 \n {\bf Definition.} Let $X$ be a normal projective surface with quotient
 singularities. Let $\cX\to\Delta$ (or $\cX/\Delta$) be a flat family of
 projective surfaces over a small disk $\Delta$. The one-parameter
 family of surfaces $\cX\to\Delta$ is called a {\it $\mQ$-Gorenstein smoothing}
 of $X$ if it satisfies the following three conditions;

\n(i) the general fiber $X_t$ is a smooth projective surface,

\n(ii) the central fiber $X_0$ is $X$,

\n(iii) the canonical divisor $K_{\cX/\Delta}$ is $\mQ$-Cartier.

\m

A $\mQ$-Gorenstein smoothing for a germ of a quotient singularity
$(X_0, 0)$ is defined similarly. A quotient singularity which
admits a $\mQ$-Gorenstein smoothing is called a {\it singularity
of class T}.

\begin{proposition}[\cite{KSB, Man91, Wa1}]
\label{pro-2.1}
 Let $(X_0, 0)$ be a germ of two dimensional quotient
 singularity. If $(X_0, 0)$ admits a $\mQ$-Gorenstein smoothing over
 the disk, then $(X_0, 0)$ is either a rational double point or a
 cyclic quotient singularity of type $\frac{1}{dn^2}(1, dna-1)$ for some
 integers $a, n, d$ with $a$ and $n$ relatively prime.
\end{proposition}

\begin{proposition}[\cite{KSB, Man91, Wa2}]
\label{pro-2.2}
\begin{enumerate}
 \item The singularities ${\overset{-4}{\circ}}$ and
 ${\overset{-3}{\circ}}-{\overset{-2}{\circ}}-{\overset{-2}{\circ}}-\cdots-
 {\overset{-2}{\circ}}-{\overset{-3}{\circ}}$ are of class $T$.
 \item If the singularity
 ${\overset{-b_1}{\circ}}-\cdots-{\overset{-b_r}{\circ}}$ is of
 class $T$, then so are
 $${\overset{-2}{\circ}}-{\overset{-b_1}{\circ}}-\cdots-{\overset{-b_{r-1}}
 {\circ}}- {\overset{-b_r-1}{\circ}} \quad\text{and}\quad
 {\overset{-b_1-1}{\circ}}-{\overset{-b_2}{\circ}}-\cdots-
 {\overset{-b_r}{\circ}}-{\overset{-2}{\circ}}.$$ \item Every
 singularity of class $T$ that is not a rational double point can be
 obtained by starting with one of the singularities described in
 $(1)$ and iterating the steps described in $(2)$.
\end{enumerate}
\end{proposition}

 Let $X$ be a normal projective surface with singularities of class
 $T$. Due to the result of Koll\'ar and Shepherd-Barron \cite{KSB},
 there is a $\mQ$-Gorenstein smoothing locally for each singularity
 of class $T$ on $X$.
 The natural question arises whether this local $\mQ$-Gorenstein
 smoothing can be extended over the global surface $X$ or not.
 Roughly geometric interpretation is the following:
 Let $\cup V_{\alpha}$ be an open covering of $X$ such that
 each $V_{\alpha}$ has at most one singularity of class $T$.
 By the existence of a local $\mQ$-Gorenstein
 smoothing, there is a $\mQ$-Gorenstein smoothing
 $\cV_{\alpha}/\Delta$. The question is if these families glue to
 a global one. The answer can be obtained by figuring out the
 obstruction map of the sheaves of deformation
 $T^i_X=Ext^i_X(\Omega_X,\cO_X)$ for $i=0,1,2$.
 For example, if $X$ is a smooth surface,
 then $T^0_X$ is the usual holomorphic tangent sheaf $T_X$ and
 $T^1_X=T^2_X=0$. By applying the standard result of deformations
 \cite{LS, Pal} to a normal projective surface with quotient
 singularities, we get the following

\begin{proposition}[\cite{Wa1}, \S 4]
\label{pro-2.3}
 Let $X$ be a normal
 projective surface with quotient singularities. Then
\begin{enumerate}
 \item The first order deformation space of $X$ is represented by
 the global Ext 1-group $\T^1_X=\Ext^1_X(\Omega_X, \cO_X)$. \item
 The obstruction lies in the global Ext 2-group
 $\T^2_X=\Ext^2_X(\Omega_X, \cO_X)$.
 \end{enumerate}
\end{proposition}

 Furthermore, by applying the general result of local-global spectral
 sequence of ext sheaves (\cite{Pal}, \S 3) to deformation theory
 of surfaces with quotient singularities so that
 $E_2^{p, q}=H^p(T^q_X) \Rightarrow \T^{p+q}_X$,
 and by $H^j(T^i_X)=0$ for $i, j\ge 1$, we also get

\begin{proposition}[\cite{Man91, Wa1}]
\label{pro-2.4}
 Let $X$ be a normal projective surface with quotient singularities.
 Then
 \begin{enumerate}
 \item We have the exact sequence
 $$0\to H^1(T^0_X)\to \T^1_X\to \ker [H^0(T^1_X)\to H^2(T^0_X)]\to 0$$
 where $H^1(T^0_X)$ represents the first order deformations of $X$
 for which the singularities remain locally a product.
 \item If $H^2(T^0_X)=0$, every local deformation of
 the singularities may be globalized.
 \end{enumerate}
\end{proposition}

 The vanishing $H^2(T^0_X)=0$ can be
 obtained via the vanishing of $H^2(T_V(-\log E))$,
 where $V$ is the minimal resolution of $X$ and
 $E$ is the reduced exceptional divisors.

\begin{theorem}[\cite{LP}]
\label{thm-2.1}
 Let $X$ be a normal projective surface with
 singularities of class $T$.
 Let $\pi: V\to X$ be the minimal resolution and
 let $E$ be the reduced exceptional
 divisors. Suppose that $H^2(T_V(-\log \ E))=0$.
 Then $H^2(T^0_X)=0$ and there is a $\mQ$-Gorenstein smoothing of $X$.
\end{theorem}

 Note that Theorem~\ref{thm-2.1} above can be
 easily generalized to any log resolution of $X$
 by keeping the vanishing
 of cohomologies under blowing up at the points.
 It is obtained by the following well-known result.

\begin{proposition} [\cite{FZ}, \S1]
\label{pro-2.5}
 Let $V$ be a nonsingular surface and let $D$ be a
 simple normal crossing divisor in $V$.
 Let $f: V'\to V$ be a blowing up of $V$ at a point p of $D$.
 Set $D'=f^{-1}(D)_{red}$.
 Then $h^2(T_{V'}(-\log \ D'))=h^2(T_V(-\log \ D))$.
\end{proposition}

\b

\section{The main construction}
\label{sec-3}

 We begin with a special elliptic fibration
 $g: E(1)=\mP^2\sharp 9\overline{\mP}^2 \to \mP^1$
 which is constructed as follows:
 Let $A$ be a line and $B$ be a smooth conic in $\mP^2$
 such that $A$ and $B$ meet at two different points.
 Choose a tangent line $L_1$ to $B$ at a point $p \in B$ so that
 $L_1$ intersects with $A$ at a different point $q \in A$, and
 draw a tangent line $L_2$ from $q$ to $B$ which tangents
 at the point $r \in B$.
 Let $L_3$ be the line connecting $p$ and $r$ which meets $A$ at $s$.
 We may assume that $p, r \not\in A \cap B$ and $s \not\in B$
 (Figure~\ref{fig-pencil-1}).
\begin{figure}[hbtb]
 \begin{center}
 \setlength{\unitlength}{1mm}
 \includegraphics[height=3cm]{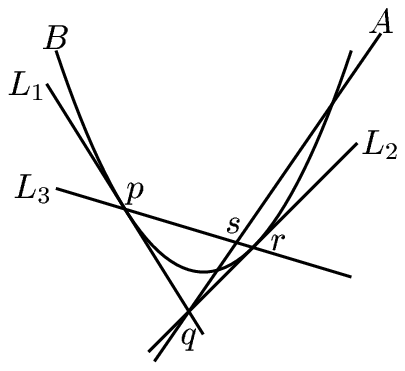}
 \end{center}
 \vspace*{-1 em}
 \caption{A cubic pencil}
 \label{fig-pencil-1}
\end{figure}

 We now consider a cubic pencil in $\mP^2$ induced by $A+B$ and
 $L_1+L_2+L_3$, i.e. $\lambda (A+B)+\mu (L_1+L_2+L_3)$,
 for $[\lambda:\mu]\in \mP^1$.
 Blow up first at $p$ and blow up at the intersection point
 of the proper transform of $B$ with the exceptional curve
 $e_1$. And then blow up again at the intersection point
 of the proper transform of $B$ with the exceptional curve
 $e_2$. Let $e_3$ be an exceptional divisor induced by the
 last blowing up. Similarly, after blowing up at $r$, blow up two
 more times at the intersection point of the proper transform of $B$
 with the exceptional curves $e_4$ and $e_5$. Let $e_6$ be the
 exceptional divisor induced by the last blowing up.
 Next, blow up at $q$, and then blow up again at the intersection point
 of the proper transform of $A$ with the exceptional curve $e_7$.
 Let $e_8$ be the exceptional curve induced by the blowing up.
 Finally,
 blowing up once at $s$, which induces the exceptional divisor $e_9$,
 we get an elliptic fibration $E(1)=\mP^2\sharp 9\overline{\mP}^2$
 over $\mP^1$.
 Let us denote this elliptic fibration by $g : Y \to \mP^1$.
 Note that there is an $I_8$-singular fiber on $g : Y \to \mP^1$ which
 consists of the proper transforms of $L_1$, $L_2$, $L_3$, $e_1$, $e_2$,
 $e_4$, $e_5$, $e_7$. There is also one $I_2$-singular fiber on $g:
 Y \to \mP^1$ which consists of the proper transforms of $A$ and $B$,
 denoted by $\tilde{A}$ and $\tilde{B}$ respectively.
 According to the list of Persson~\cite{Pers}, there exist only
 two more nodal singular fibers on $g : Y \to \mP^1$.
 For example, the pencil used above can be chosen explicitly as follows:
 $\lambda x(x^2 + (y-2z)^2-z^2) + \mu (y-\sqrt{3}x)(y+\sqrt{3}x)(2y-3z)$.
 Note that this pencil has singular fibers at $[\lambda:\mu]=[1:0],
 [0:1], [3\sqrt{3}:2]$ and $[3\sqrt{3}:-2]$.
 Hence the fibration $g : Y \to \mP^1$ has one $I_8$-singular fiber,
 one reducible $I_2$-singular fiber, and two nodal singular fibers.
 On the other hand, $e_3$, $e_6$, $e_8$ and $e_9$ in $Y=E(1)$
 are the sections of $g : Y \to \mP^1$ such that
 two sections $e_3$ and $e_6$ connect the proper transform $\tilde{B}$
 of $B$ and the proper transforms of $e_2$, $e_5$, respectively, and
 the other two sections $e_8$ and $e_9$ connect the proper transform
 $\tilde{A}$ of $A$ and the proper transforms of $e_7$, $L_3$, respectively.
 Among these four sections, we will use only three sections, $e_3$,
 $e_8$ and $e_9$, in the following main construction. We denote the
 three sections  $e_3$, $e_8$, $e_9$ by $S_1$, $S_2$, $S_3$,
 respectively (Figure~\ref{fig-Y-1}).

\begin{figure}[hbtb]
 \begin{center}
 \setlength{\unitlength}{1mm}
 \includegraphics[height=4cm]{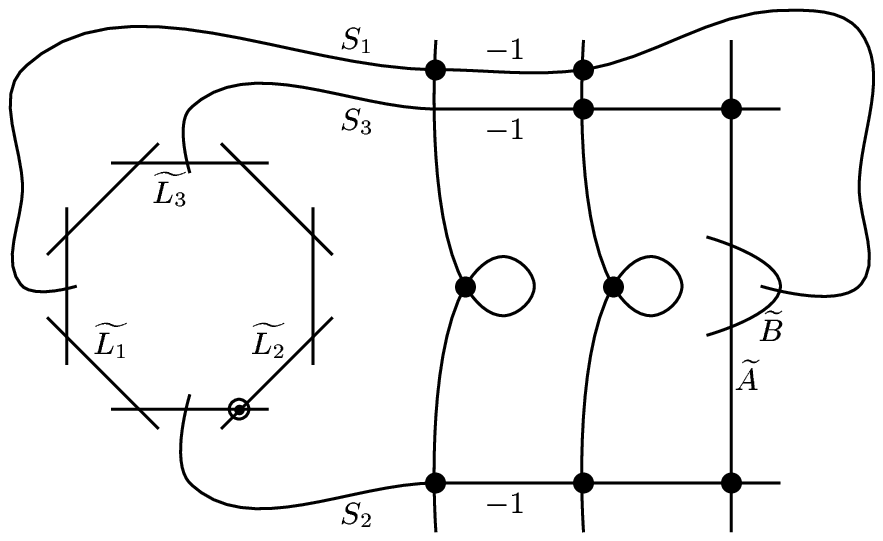}
 \end{center}
 \vspace*{-1.5 em}
 \caption{An elliptic fibration $Y=E(1)$}
 \label{fig-Y-1}
\end{figure}

\subsection*{Main Construction}

 Let $Z':=Y\sharp 2{\overline \mP}^2$ be
 the surface obtained by blowing up at two singular points of two
 nodal fibers on $Y$, and denote this map by $\tau$. Then
 there are two fibers such that each consists of two $\mP^1$s,
 say $E_i$ and $F_i$, satisfying $E_i^2=-1$, $F_i^2=-4$ and
 $E_i\cdot F_i=2$ for $i=1,2$.
 Note that each $E_i$ is an exceptional curve and
 $F_i$ is the proper transform of a nodal fiber.
 We blow up twice at the intersection points between $S_1$ and $F_i$
 for $i=1, 2$. We also blow up twice at the intersection points between
 $S_3$ and $F_2$, $\tilde{A}$. And then blow up three times at the
 intersection points of between $S_2$ and $F_1$, $F_2$, $\tilde{A}$.
 Finally, blowing up at the marked point $\bigodot$ on the $I_8$-singular
 fiber, we then get $Z'':=Y\sharp 10{\overline \mP}^2$
 (Figure~\ref{fig-Zdp-1}).
 Note that the self-intersection numbers
 of proper transforms are as follows: $[S_1]^2=-3$, $[S_2]^2=-4$,
 $[S_3]^2=-3$, $[F_1]^2=-6$, $[F_2]^2=-7$ and $[\tilde{A}]^2=-4$.
 Here we denote the proper transforms of $S_i$, $F_j$, $\tilde{A}$,
 $\tilde{B}$ again by the same notations.

\begin{figure}[hbtb]
 \begin{center}
 \setlength{\unitlength}{1mm}
 \includegraphics[height=4.5cm]{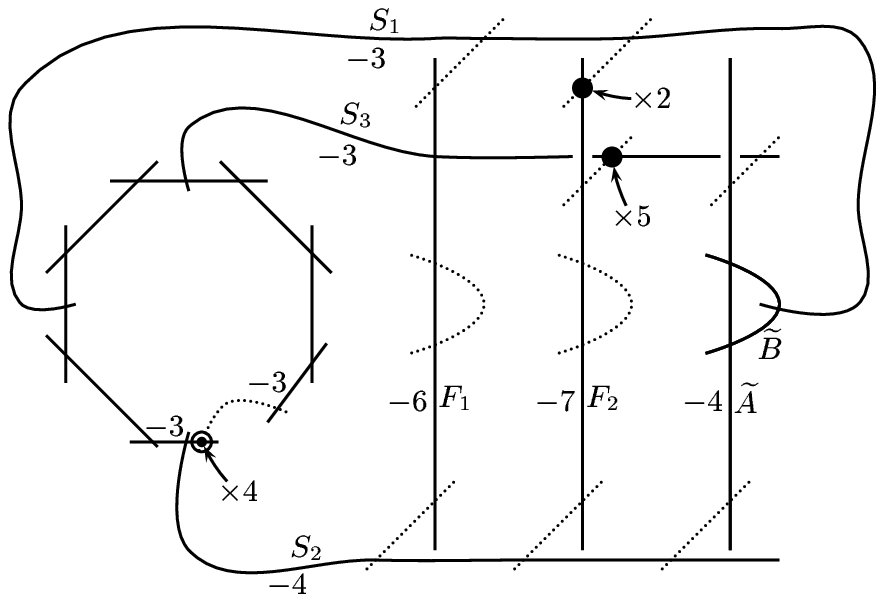}
 \end{center}
 \vspace*{-1 em}
 \caption{A rational surface $Z''=Y\sharp 10{\overline \mP}^2$}
 \label{fig-Zdp-1}
\end{figure}

 Next, we blow up two times successively at the intersection point
 between $F_2$ and the exceptional curve meeting $F_2$ and $S_1$.
 It makes a chain of $\mP^1$, $\nc{-9}-\nc{-1}-\nc{-2}-\nc{-2}$,
 lying in the total transform of $F_2$. We then blow up five times
 successively at the intersection point between $S_3$ and the
 exceptional curve meeting $S_3$ and $F_2$, so that it produces a
 chain of $\mP^1$,
 $\nc{-8}-\nc{-1}-\nc{-2}-\nc{-2}-\nc{-2}-\nc{-2}-\nc{-2}$,
 lying in the total transform of $S_3$.
 Finally, we blow up four times successively at the marked point
 $\bigodot$ of the proper transform of the $I_8$-singular fiber,
 so that it produces a chain of $\mP^1$,
 $\nc{-7}-\nc{-1}-\nc{-2}-\nc{-2}-\nc{-2}-\nc{-2}-\nc{-3}$,
 lying in the proper transform of the $I_8$-singular fiber.
 Note that the self-intersection numbers of proper transforms, denoted
 again by the same notations, are as follows:
 $[S_1]^2=-3$, $[S_2]^2=-4$, $[S_3]^2=-8$,
 $[F_1]^2=-6$, $[F_2]^2=-9$ and $[\tilde{A}]^2=-4$.

\begin{figure}[hbtb]
 \begin{center}
 \setlength{\unitlength}{1mm}
 \includegraphics[height=5cm]{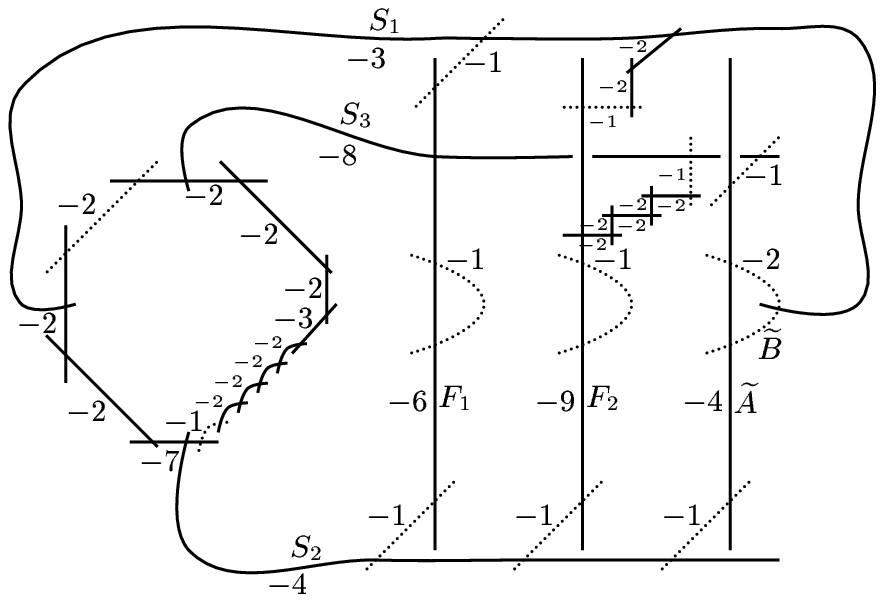}
 \end{center}
 \vspace*{-1 em}
 \caption{A rational surface $Z=Y\sharp 21{\overline \mP}^2$}
 \label{fig-Z-1}
\end{figure}

 In summary, we get a rational surface $Z:= Y\sharp 21{\overline \mP}^2$,
 which contains four disjoint linear chains of
 $\mP^1$: $C_{2,1}= \nc{-4}$ ($\tilde{A}$),
 $C_{7,1}= \nc{-9}-\nc{-2}-\nc{-2}-\nc{-2}-\nc{-2}-\nc{-2}$
 (which contains the proper transform of $F_2$),
 $C_{19,5}= \nc{-4}-\nc{-7}-\nc{-2}-\nc{-2}-\nc{-3}-\nc{-2}-\nc{-2}$
 (which contains the proper transforms of $S_1$, $S_2$, and a part of
 proper transforms of $I_8$-singular fibers) and
 $C_{35,6}= \nc{-6}-\nc{-8}-\nc{-2}-\nc{-2}-\nc{-2}-\nc{-3}-\nc{-2}
  -\nc{-2}-\nc{-2}-\nc{-2}$ (which contains the proper transforms of
 $S_3$, $F_1$, and a part of proper transforms of $I_8$-singular fibers)
 (Figure~\ref{fig-Z-1}).

 Finally, we follow the same procedures as in~\cite{LP}. That is,
 we contract these four disjoint chains of $\mP^1$ from $Z$.
 Since it satisfies the Artin's criterion, it produces a projective
 surface with four singularities of class $T$ (\cite{Artin}, \S 2).
 We denote this surface by $X$.
 In the next sections,
 we are going to prove that $X$ has a $\mQ$-Gorenstein smoothing and
 a general fiber of the $\mQ$-Gorenstein smoothing is a simply connected
 minimal complex surface of general type with $p_g=0$ and $K^2=3$.

\m

 In the remaining of this section,
 we investigate a rational blow-down manifold of the surface $Z$
 obtained in the main construction above.
 First we describe topological aspects of a rational blow-down surgery
 (\cite{FS, P1} for details):
 For any relatively prime integers $p$ and $q$ with $p > q > 0$,
 we define a configuration $C_{p,q}$ as a smooth $4$-manifold obtained
 by plumbing disk bundles over the $2$-sphere instructed by the following
 linear diagram
 $$\underset{u_{k}}{\overset{-b_k}{\circ}}-\underset{u_{k-1}}
 {\overset{-b_{k-1}}{\circ}}-\cdots -\underset{u_2}{\overset{-b_{2}}{\circ}}
 -\underset{u_1}{\overset{-b_{1}}{\circ}}$$

\n where $\frac{p^{2}}{pq-1} =[b_{k},b_{k-1}, \ldots, b_{1}]$ is the
 unique continued fraction with all $b_{i} \geq 2$, and  each
 vertex $u_{i}$ represents a disk bundle over the $2$-sphere whose
 Euler number is $-b_{i}$. Orient the $2$-spheres in $C_{p,q}$ so that
 $u_{i}\cdot u_{i+1} = +1$. Then the configuration $C_{p,q}$
 is a negative definite simply connected smooth $4$-manifold whose boundary
 is the lens space $L(p^2, 1-pq)$.

\m

\n {\bf Definition.} Suppose $M$ is a smooth $4$-manifold
 containing a configuration $C_{p,q}$. Then we construct
 a new smooth $4$-manifold $M_p$,
 called a {\em $($generalized$)$ rational blow-down} of $M$,
 by replacing $C_{p,q}$ with the rational ball $B_{p,q}$.
 Note that this process is well-defined, that is, a new smooth
 $4$-manifold $M_{p}$ is uniquely determined
 (up to diffeomorphism) from $M$ because each diffeomorphism
 of $\partial B_{p,q}$ extends over the rational ball $B_{p,q}$.
 We call this a {\em rational blow-down} surgery.
 Furthermore, M. Symington proved that a rational blow-down manifold
 $M_{p}$ admits a symplectic structure in some cases.
 For example, if $M$ is a symplectic $4$-manifold containing
 a configuration $C_{p,q}$ such that all $2$-spheres $u_i$ in $C_{p,q}$
 are symplectically embedded and intersect positively,
 then the rational blow-down manifold $M_{p}$ also admits a symplectic
 structure~\cite{Sy1, Sy2}.

\m

 Now we perform a rational blow-down surgery of the surface $Z$
 obtained in the main construction.
 Note that the surface $Z$ contains four disjoint configurations
 - $C_{35,6},\, C_{19,5},\, C_{7,1}$ and $C_{2,1}$.
 Let us decompose the surface $Z$ into
 \[Z= Z_{0}\cup\{C_{35,6}\cup C_{19,5} \cup C_{7,1} \cup C_{2,1}\}.\]
 Then the $4$-manifold, say  $Z_{35,19,7,2}$, obtained by
 rationally blowing down along the four configurations can be decomposed into
 \[Z_{35,19,7,2}= Z_{0}\cup\{B_{35,6}\cup B_{19,5}\cup B_{7,1}\cup B_{2,1}\},\]
 where $B_{35,6},\, B_{19,5},\, B_{7,1}$ and $B_{2,1}$
 are the corresponding rational balls.  We claim that

\begin{theorem}
\label{thm-3.1}
  The rational blow-down $Z_{35,19,7,2}$ of the surface $Z$
  in the main construction is a simply connected closed symplectic
  $4$-manifold with $b_2^+=1$ and $K^2 =3$.
\end{theorem}

\begin{proof}
 Since all the curves lying in the configurations
 $C_{35,6}$, $C_{19,5}$, $C_{7,1}$ and $C_{2,1}$
 are symplectically (in fact, holomorphically) embedded $2$-spheres,
 Symington's result~\cite{Sy1, Sy2} guarantees the existence of a symplectic
 structure on the rational blow-down $4$-manifold $Z_{35,19,7,2}$.
 Furthermore, it is easy to check that
 $b_2^+(Z_{35,19,7,2})=b_2^+(Z)=1$ and
 $K^2(Z_{35,19,7,2}) = K^2(Z) + 24 = 3$.

 It remains to prove the simple connectivity of $Z_{35,19,7,2}$:
 Since $\pi_{1}(\partial B_{p,q}) \rightarrow \pi_{1}(B_{p,q})$
 is surjective (\cite{LW}, \S5), by Van-Kampen theorem,
 it suffices to show that $\pi_{1}(Z_0) =1$.
 First, note that $Z$ and all four configurations $C_{35,6}$,
 $C_{19,5}$, $C_{7,1}$ and $C_{2,1}$ are all simply connected.
 Hence, applying Van-Kampen theorem on $Z$ inductively, we get
 \begin{eqnarray}
 1 = \pi_{1}(Z_0)/< N_{i_{*}(\alpha)}, N_{{j_1}_{*}(\beta_1)},
    N_{{j_2}_{*}(\beta_2)}, N_{{j_3}_{*}(\beta_3)}> .
 \end{eqnarray}
 Here $i_{*}$, ${j_{1}}_{*}$, ${j_{2}}_{*}$ and ${j_{3}}_{*}$
 are induced homomorphisms by inclusions
 $i: \partial C_{19,5} \rightarrow Z_0$,
 $j_1: \partial C_{35,6} \rightarrow Z_0$,
 $j_2: \partial C_{7,1} \rightarrow Z_0$ and
 $j_3: \partial C_{2,1} \rightarrow Z_0$ respectively.
 We may also choose the generators, say $\alpha$, $\beta_1$,
 $\beta_2$ and $\beta_3$, of
 $\pi_{1}(\partial C_{19,5}) \cong {\mZ}_{19^2}$,
 $\pi_{1}(\partial C_{35,6}) \cong {\mZ}_{35^2}$,
 $\pi_{1}(\partial C_{7,1}) \cong {\mZ}_{7^2}$ and
 $\pi_{1}(\partial C_{2,1}) \cong {\mZ}_{2^2}$,
 so that $\alpha$, $\beta_1$, $\beta_2$ and $\beta_3$
 are represented by circles $\partial C_{19,5} \cap E'_1$
 (equivalently $\partial C_{19,5} \cap E'_2$ or $\partial C_{19,5} \cap E'_3$),
 $\partial C_{35,6} \cap E'_1$,
 $\partial C_{7,1} \cap E'_2$ and
 $\partial C_{2,1} \cap E'_3$, respectively,
 where $E'_1$, $E'_2$ and $E'_3$ are exceptional curves connecting
 the last $2$-spheres in the configurations $C_{19,5}$ and $C_{35,6}$,
 $C_{19,5}$ and $C_{7,1}$, $C_{19,5}$ and $C_{2,1}$, respectively.
 Note that the circle cut out by a $2$-sphere which intersects transversely
 one of the two end $2$-spheres in the configurations $C_{p, q}$ is a generator
 of $\pi_1$ of the lens space, and other circles cut out by a $2$-sphere which
 intersects transversely one of the middle $2$-spheres in the configurations
 $C_{p, q}$, is a power of the generator \cite{Mu}.
 Finally $N_{i_{*}(\alpha)}$ and $N_{{j_i}_{*}(\beta_i)}$
 denote the least normal subgroups of $\pi_{1}(Z_0)$ containing
 $i_{*}(\alpha)$ and ${j_i}_{*}(\beta_i)$ respectively.
 Note that there is a relation between
 $i_{*}(\alpha)$ and ${j_i} _{*}(\beta_i)$ when we restrict them to
 $Z_0$. That is, they satisfy either
 $i_{*}(\alpha) = \gamma^{-1} \cdot {j_i}_{*}(\beta_i)\cdot \gamma$ or
 $i_{*}(\alpha) = \gamma^{-1} \cdot {j_i}_{*}(\beta_i)^{-1} \cdot \gamma$
 (depending on orientations) for some path $\gamma$,
 because one is homotopic to the other in
 $E'_i \setminus \{\mathrm{two\, \, open\, \, disks} \} \subset Z_0$.
 Hence, by combining two facts above, for example, we get
 ${i_{*}(\alpha)}^{19^2} = (\gamma^{-1} \cdot {j_1}_{*}(\beta_1)^{\pm 1} \cdot
 \gamma)^{19^2} = \gamma^{-1} \cdot {j_1}_{*}(\beta_1)^{\pm 19^2} \cdot \gamma
  = 1 = {{j_1}_{*}(\beta_1)}^{35^2}$. Since the two numbers $19^2$ and $35^2$
  are relatively prime, the element ${{j_1}_{*}(\beta_1)}$ should
  be trivial. So the relation
  $i_{*}(\alpha) = \gamma^{-1} \cdot {j_1}_{*}(\beta_1)^{\pm 1} \cdot \gamma$
  implies the triviality of $i_{*}(\alpha)$.

  Furthermore, since $i_{*}(\alpha)$ and ${j_i}_{*}(\beta_i)$
  are also conjugate to each other for $i=2,3$,
  the triviality of $i_{*}(\alpha)$  implies that
  ${{j_2}_{*}(\beta_2)}$ and ${{j_3}_{*}(\beta_3)}$  are trivial.
  Hence,
  all normal subgroups $N_{i_{*}(\alpha)}$ and $N_{j_{i_{*}}(\beta_i)}$
  are trivial, so that relation (1) implies $\pi_{1}(Z_0) =1$.
\end{proof}

\b

\section{Existence of smoothing}
\label{sec-4}

 In this section we prove the existence of a $\mQ$-Gorenstein smoothing
 for the singular projective surface $X$ which is obtained by contracting
 from the rational surface $Z$ in the main construction in Section 3.
 The procedure is exactly parallel to the $K^2=2$ case appeared in~\cite{LP}.
 For the completeness of this article, we repeat the procedure here.
 First we need the following two essential lemmas.

\begin{lemma}[\cite{LP}]
\label{lem-4.1}
 Let $Y$ be a rational elliptic surface. Let $C$ be a
 general fiber of the elliptic fibration $g: Y\to\mP^1$.
 Then the global sections $H^0(Y, \Omega_Y(kC))$ are coming from
 the global sections $H^0(Y, g^*\Omega_{\mP^1}(k))$.
 In particular, $h^0(Y, \Omega_Y(kC))=k-1$ for $k\ge 1$.
\end{lemma}

\begin{lemma}
\label{lem-4.2}
 Suppose $Z'=Y\sharp 2{\overline \mP}^2$ is the rational elliptic surface
 in the main construction and $F_i$ is the proper transform
 of a nodal fiber in $Z'$ for $i\,=\,1,2$.
 Let $\tilde{A}$ and $\tilde{B}$ be the proper transforms of the line $A$
 and the conic $B$ respectively.
 Let $D$ be the reduced subscheme of the $I_8$-singular fiber.
 Assume that $D$ is not whole $I_8$-singular fiber as a reduced scheme.
 Then $H^2(Z', T_{Z'}(-F_1-F_2-\tilde{A}-D))=0$.
\end{lemma}

 \begin{proof}
 By the Serre duality, it is equal to prove
 $H^0(Z', \Omega_{Z'}(K_{Z'}+F_1+F_2+\tilde{A}+D))=0$.
 Let $C$ be a general fiber in the elliptic fibration $g: Y\to\mP^1$.
 Since $K_{Z'}=\tau^*(-C)+E_1+E_2$ and $\tau^*(C)=F_1+2E_1=F_2+2E_2$,
 $H^0(Z', \Omega_{Z'}(K_{Z'}+F_1+F_2+\tilde{A}+D))\subseteq
 H^0(Z', \Omega_{Z'}(\tau^*(C)+\tilde{A}+D))$.
 Furthermore, since $\tilde{A}$ and $D$ are not changed by the map $\tau$,
 we have the same curves in $Y$.
 Then $H^0(Z',\Omega_{Z'}(\tau^*(C)+\tilde{A}+D))=H^0(Y,\Omega_Y(C+\tilde{A}+D))$
 by the projection formula.  We note that $\tau_*\Omega_{Z'}=\Omega_Y$.
 Then the cohomology $H^0(Y,\Omega_Y(C+\tilde{A}+D))$ vanishes:
 We note that $H^0(Y,\Omega_Y(C+\tilde{A}+D))=H^0(Y,\Omega_Y(3C-\tilde{B}-G))$
 with $G+D =I_8$-singular fiber.
 By Lemma~\ref{lem-4.1} above, all global sections of $\Omega_Y(3C)$ are
 coming form the global sections of $g^*(\Omega_{\mP^1}(3))=g^*(\cO_{\mP^1}(1))$.
 But, if this global section vanishes on $\tilde{B}$ and $G$ which lie
 on two different fibers, then it should be zero.
 Note that the dualizing sheaf of each fiber of the elliptic fibration is
 the structure sheaf of the fiber by using the adjunction formula.
 Therefore we have $H^2(Z', T_{Z'}(-F_1-F_2-\tilde{A}-D))=0$.
\end{proof}

\begin{theorem}
\label{thm-4.1}
 The projective surface $X$ with four singularities of class T
 in the main construction has a $\mQ$-Gorenstein smoothing.
\end{theorem}

\begin{proof}
 Let $D$ be the reduced scheme of the $I_8$-singular fiber minus
 the rational $-2$-curve $G$ in the main construction in Section 3.
 Note that the curve $G$ is not contracted from $Z$ to $X$.
 By Lemma~\ref{lem-4.2} above,
 we have $H^2(Z', T_{Z'}(-F_1-F_2-\tilde{A}-D))=
 H^2(Z', T_{Z'}(-\log(F_1+F_2+\tilde{A}+D)))=0$.
 Let $D_{Z'}=F_1+F_2+\tilde{A}+D+S_1+S_2+S_3$.
 Since the self-intersection number of the section is $-1$,
 we still have the vanishing $H^2(Z', T_{Z'}(-\log \ D_{Z'}))=0$.
 Remind that the surface $Z''=Y\sharp 10{\overline \mP}^2$ is
 obtained by blowing up eight times from $Z'$:
 We blow up twice at the intersection points between $S_1$ and $F_i$
 for $i=1, 2$. We also blow up twice at the intersection points between
 $S_3$ and $F_2$, $\tilde{A}$. And then blow up three times at the
 intersection points of between $S_2$ and $F_1$, $F_2$, $\tilde{A}$.
 Finally, blowing up at the marked point $\bigodot$ on the $I_8$-singular
 fiber, we then get a rational surface $Z''$ (Figure~\ref{fig-Zdp-1}).
 Now choose the exceptional curve in the total transform of
 $I_8$-singular fiber which intersects the proper transform of $I_8$,
 and choose two exceptional curves in the total transform of $F_2$
 which intersect the proper transform of $S_1$ and $S_3$.
 Let $D_{Z''}$ be the reduced scheme of $F_1+F_2+\tilde{A}+D+S_1+S_2+S_3+$
 these three exceptional divisors.
 Then, by Lemma~\ref{lem-4.2}, Proposition~\ref{pro-2.5},
 and the self-intersection number, $-1$, of each exceptional divisor,
 we have $H^2(Z'', T_{Z''}(-\log \ D_{Z''}))=0$.
 Finally, by using the same argument finite times through blowing up,
 we have the vanishing $H^2(Z, T_{Z}(-\log \ D_{Z}))=0$,
 where $D_{Z}$ are the four disjoint linear chains of $\mP^1$ which are
 the exceptional divisors from the contraction from $Z$ to $X$.
 Hence there is a $\mQ$-Gorenstein smoothing for $X$ by Theorem~\ref{thm-2.1}.
\end{proof}

\b

\section{Properties of $X_t$}
\label{sec-5}

 We showed in Section 4 that the projective surface $X$ has a
 $\mQ$-Gorenstein smoothing. We denote a general fiber of the
 $\mQ$-Gorenstein smoothing by $X_t$. In this section, we prove that
 $X_t$ is a simply connected and minimal surface of general type with
 $p_g=0$ and $K_{X_t}^2=3$ by using a standard argument.
 Of course, the procedure is also exactly parallel to the $K^2=2$ case
 appeared in~\cite{LP}.

 We first prove that $X_t$ satisfies $p_g=0$ and $K^2=3$:
 Since $Z$ is a nonsingular rational surface and $X$ has
 only rational singularities, $X$ is a projective surface with
 $H^1(X,\mathcal{O}_X)=H^2(X,\mathcal{O}_X)=0$. Then the upper
 semi-continuity implies that $H^2(X_t,\mathcal{O}_{X_t})=0$,
 so that the Serre duality implies that $p_g(X_t)=0$.
 And $K_X^2=3$ can be computed by using the explicit description of
 $f^{\ast}{K_X}$ (refer to Equation~\eqref{equation:f^*K_X} below).
 Then we have $K_{X_t}^2=3$ by the property of the $\mQ$-Gorenstein
 smoothing.

 Next, let us show the minimality of $X_t$: As we noticed in
 Section 3, the surface $Z$ contains the following four chains of
 $\mP^1$ including the proper transforms of three sections. We denote
 them by the following dual graphs
\begin{eqnarray*}
  C_{35,6} &=& \ntc{-6}{G_1}-\ntc{-8}{G_2}-\ntc{-2}{G_3}-\ntc{-2}{G_4}
                 -\ntc{-2}{G_5}-\ntc{-3}{G_6}-\ntc{-2}{G_7}-\ntc{-2}{G_8}
                 -\ntc{-2}{G_9}-\ntc{-2}{G_{10}}, \\
  C_{19,5} &=& \ntc{-4}{H_1}-\ntc{-7}{H_2}-\ntc{-2}{H_3}-\ntc{-2}{H_4}
                 -\ntc{-3}{H_5}-\ntc{-2}{H_6}-\ntc{-2}{H_7}, \\
  C_{7,1} &=& \ntc{-9}{I_1}-\ntc{-2}{I_2}-\ntc{-2}{I_3}-\ntc{-2}{I_4}
                -\ntc{-2}{I_5}-\ntc{-2}{I_6}, \ \ \ C_{2,1}= \ntc{-4}{\tilde{A}}
 \end{eqnarray*}
 and we also denote the four special fibers by the following dual graphs:
\begin{gather*}
\begin{matrix}
 & & \uc{\tilde{B}, -2} & & \\ %
 & & \parallel & & \\ %
     \dc{E_4, -1} & - & \dc{\tilde{A}, -4} & - & \dc{E_4', -1}
\end{matrix} \qquad  \qquad
\begin{matrix}
 & & \uc{E_1, -1} & & \\ %
 & & \parallel & & \\ %
     \dc{E_1',-1} & - & \dc{G_1, -6} & - & \dc{E_1'', -1}
\end{matrix} \\[2em] %
\end{gather*}
\begin{gather*}
\begin{array}{ccccccccccccc}
 & & \uc{E_2,-1} \\ %
 & & \parallel \\ %
     \dc{E_2''',-1} & - & \dc{I_1,-9} & - & \dc{E_2',-1} & - &
     \dc{H_7,-2}
 & - & \dc{H_6,-2}\\ %
 & & \vert \\ %
 & & \dc{I_2,-2} & - & \dc{I_3,-2} & - & \dc{I_4,-2} & - & \dc{I_5,-2}
  & - & \dc{I_6,-2} & - & \dc{E_2'',-1} \\ %
\end{array}\\[2em] %
\begin{array}{ccccccccccccc}
 \uc{G,-2} & - & \uc{G_3,-2} & - & \uc{G_4,-2} & - & \uc{G_5,-2}
 & - & \uc{G_6,-3} & - & \uc{G_7,-2} & - & \uc{G_8,-2} \\ %
 & \diagdown& & & & & & & & & & & \vert \\ %
 & & \uc{H_4,-2} & - & \uc{H_3,-2} & - & \uc{H_2,-7} & - & \uc{E_3,-1}
  & - & \uc{G_{10},-2} & - & \uc{G_9,-2}
\end{array}
\end{gather*}

 Note that the final one indicates the total transform of $I_8$-singular
 fiber and $G$ denotes the rational $-2$-curve which is not contracted from
 $Z$ to $X$. The numbers indicate the self-intersection numbers of curves.
 Let $f : Z \to X$ and let $h : Z \to Y$. Then we have

\begin{align*}
&\begin{aligned}
 K_Z \equiv f^{\ast}{K_X} - &\left( \frac{29}{35} G_1 + \frac{34}{35} G_2
 + \frac{33}{35} G_3 + \frac{32}{35} G_4 + \frac{31}{35} G_5
 + \frac{30}{35} G_6 + \frac{24}{35} G_7 + \frac{18}{35} G_8 \right.\\ %
 &+ \frac{12}{35} G_9 + \frac{6}{35} G_{10}
 + \frac{14}{19} H_1 + \frac{18}{19} H_2 + \frac{17}{19} H_3
 + \frac{16}{19} H_4 + \frac{15}{19} H_5 + \frac{10}{19} H_6 \\ %
 &+\left. \frac{5}{19} H_7
 + \frac{6}{7} I_1 + \frac{5}{7} I_2 + \frac{4}{7} I_3
 + \frac{3}{7} I_4 + \frac{2}{7} I_5 + \frac{1}{7} I_6 +\frac{1}{2}
 \tilde{A}\right),
\end{aligned}\\ %
&\begin{aligned}
 K_Z \equiv h^{\ast}{K_Y} &+ E_1 + E_1' + E_1'' \\ %
 &+ E_2 + 3E_2' + 2H_7 + H_6 + 6E_2'' + 5I_6 + 4I_5 + 3I_4 + 2I_3 +
 I_2 + E_2''' \\ %
 &+ 5E_3 + 4G_{10} + 3G_9 + 2G_8 + G_7 \\ %
 &+ E_4 + E_4'.
\end{aligned}\\ %
\end{align*}

 On the other hand, we have

\begin{align*}
&\begin{aligned}
 h^{\ast}{K_Y} \equiv &-\frac{1}{2} \left( 2E_1 + E_1' + E_1'' + G_1 \right) \\ %
 &-\frac{1}{2}\left( 2E_2 + 3E_2' + 2H_7 + H_6 + 6E_2''
 + 5I_6 + 4I_5 + 3I_4 + 2I_3 + I_2 + E_2''' + I_1\right)
\end{aligned}
\end{align*}

 Hence, combining these relations, we get

\begin{equation}
\label{equation:f^*K_X}
\begin{split}
 f^{\ast}{K_X} \equiv &\ \frac{1}{2}{E_1'}+\frac{1}{2}{E_1''}+
 \frac{3}{2}{E_2'}+3{E_2''}+\frac{1}{2}{E_2'''}+5{E_3}+{E_4}+{E_4'}
 + \frac{23}{70}{G_1} + \frac{34}{35}{G_2} \\ %
 &+ \frac{33}{35}{G_3} + \frac{32}{35}{G_4} + \frac{31}{35}{G_5}
 + \frac{6}{7}{G_6} + \frac{59}{35}{G_7} + \frac{88}{35}{G_8}
 + \frac{117}{35}{G_9} + \frac{146}{35}{G_{10}} \\ %
 &+ \frac{14}{19}{H_1} + \frac{18}{19}{H_2} + \frac{17}{19}{H_3}
 + \frac{16}{19}{H_4} + \frac{15}{19}{H_5} + \frac{39}{38}{H_6}
 + \frac{24}{19}{H_7} \\ %
 &+ \frac{5}{14}{I_1} + \frac{17}{14}{I_2} + \frac{11}{7}{I_3}
 + \frac{27}{14}{I_4} + \frac{16}{7}{I_5} + \frac{37}{14}{I_6}
 + \frac{1}{2}{\tilde{A}}.
\end{split}
\end{equation}

 Since all the coefficients are positive in the expression of
 $f^{\ast}{K_X}$, the $\mQ$-divisor $f^{\ast}{K_X}$ is net if
 $f^{\ast}{K_X} \cdot E_i \ge 0$ for $i=3,4$, and $f^{\ast}{K_X}
 \cdot E_i' \ge 0$ for $i=1,2,4$, and $f^{\ast}{K_X} \cdot E_i'' \ge
 0$ for $i=1,2$, and $f^{\ast}{K_X} \cdot E_2''' \ge 0$ - We have
 $f^{\ast}{K_X} \cdot E_3 = \frac{79}{665}$, $f^{\ast}{K_X} \cdot E_4
 = \frac{33}{70}$, $f^{\ast}{K_X} \cdot E_1' = \frac{411}{665}$,
 $f^{\ast}{K_X} \cdot E_2' = \frac{16}{133}$, $f^{\ast}{K_X} \cdot
 E_4' = \frac{9}{38}$, $f^{\ast}{K_X} \cdot E_1'' = \frac{376}{665}$,
 $f^{\ast}{K_X} \cdot E_2'' = \frac{43}{70}$, and $f^{\ast}{K_X}
 \cdot E_2''' = \frac{79}{133}$. Note that other divisors are
 contracted under the map $f$. The nefness of $f^{\ast}{K_X}$ implies
 the nefness of $K_X$. Since all coefficients are positive in the
 expression of $f^{\ast}{K_X}$, we get the vanishing $h^0(-K_X)=0$.
 Hence, by the upper semi-continuity property, i.e. the vanishing
 $h^0(-K_X)=0$ implies that $h^0(-K_{X_t})=0$, we conclude that $X_t$
 is not a rational surface: If $X_t$ is a rational surface with
 $h^0(-K_{X_t})=0$, then $\chi(2K_{X_t}) \le 0$. But $\chi(2K_{X_t})
 = \chi(\mathcal{O}_{X_t}) + K_{X_t}^2 = 4$, which is a
 contradiction. Since $K_{X_t}^2=3$, $X_t$ is a surface of general
 type by the classification theory of surfaces. Let $\pi : \cX \to
 \Delta$ be a $\mQ$-Gorenstein smoothing of $X$. Since the
 $\mQ$-Cartier divisor $K_{\cX/\Delta}$ is $\pi$-big over $\Delta$ and
 $\pi$-nef at the point $0$, the nefness of $K_{X_t}$ is also
 obtained by shrinking $\Delta$ if it is necessary~\cite{Nak}.
 Therefore we have

 \begin{proposition}
\label{prop-5.1}
 $X_t$ is a minimal surface of general type with
 $p_g=0$ and $K_{X_t}^2=3$.
\end{proposition}

 Finally, applying the standard arguments about Milnor fibers
 (\cite{LW}, \S 5),
 we conclude that $X_t$ is diffeomorphic to the rational blow-down
 $4$-manifold $Z_{35,19,7,2}$ constructed in Theorem~\ref{thm-3.1}
 (see~\cite{LP} for details).
 Hence the simple connectivity of $X_t$ follows from the fact
 that $Z_{35,19,7,2}$ is simply connected.

\b

\section{More examples}
\label{sec-6}

 In this section we construct another example of simply connected,
 minimal, complex surfaces of general type with $p_g=0$ and $K^2=3$
 using a different configuration coming from a different elliptic
 pencil in $\mP^2$. Since all the proofs are basically the same as
 the case of the main example constructed in Section 3,
 we only explain how to construct it.

\subsection*{Construction}

 We first consider an elliptic fibration on $E(1)$ which
 has one $I_6$-singular fiber, two $I_2$-singular fibers, and two nodal fibers.
 Such an elliptic fibration can be constructed explicitly as follows:
 Let $A = \{[x:y:z] \in \mP^2 : x+2y+z=0 \}$ be a line and
 $B = \{[x:y:z] \in \mP^2 : x^2 + xy + yz = 0\}$ be a conic in $\mP^2$.
 Let $L_1 = \{[x:y:z] \in \mP^2 : z=0 \}$, $L_2 = \{[x:y:z] \in \mP^2 : y=0 \}$,
 $L_3 := \{[x:y:z] \in \mP^2 : x=0\}$ be three coordinate lines.
 Denote the intersection points of $L_1 + L_2 + L_3$ and $A + B$ as follows
 (Figure~\ref{fig-pencil-2}):
\begin{align*}
 p_1 &= [0:1:0] = B \cap L_1 \cap L_3, & p_2 &= [1:-1:0] = B \cap L_1, \\ %
 p_3 &= [2:-1:0] = A \cap L_1, & p_4 &= [0:0:1] = B \cap L_2 \cap L_3, \\ %
 p_5 &= [1:0:-1] = A \cap L_2, & p_6 &= [0:1:-2] = A \cap L_3.
\end{align*}

\begin{figure}[hbtb]
 \begin{center}
 \setlength{\unitlength}{1mm}
 \includegraphics[height=3.8cm]{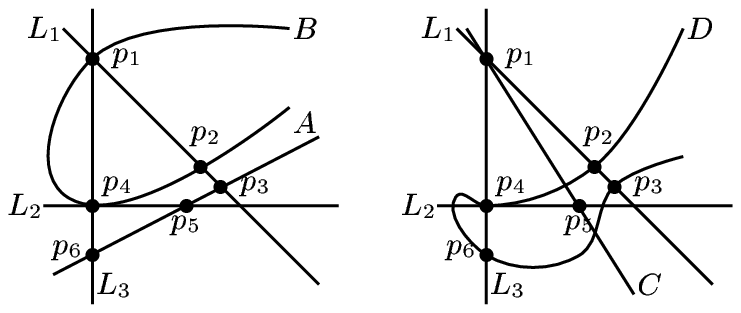}
 \end{center}
 \caption{A cubic pencil}
 \label{fig-pencil-2}
\end{figure}

 Note that $B$ is tangent to $L_2$ at $p_4$. We now consider a cubic
 pencil in $\mP^2$ induced by $A+B$ and $L_1+L_2+L_3$, i.e.,
 $\lambda(A+B) + \mu(L_1 + L_2 + L_3)$, for $[\lambda:\mu] \in \mP^1$.
 This cubic pencil has a special member. If $\lambda = 1$ and
 $\mu = 2$, then the cubic pencil $\lambda(A+B) + \mu(L_1 + L_2 +
 L_3)$ is decomposed as follows:
\begin{equation*}
 (A + B) + 2(L_1 + L_2 + L_3) = (x+z)(x^2+3xy+2y^2+yz).
\end{equation*}
 Set $C = \{[x:y:z] \in \mP^2 : x+z=0 \}$ and $D = \{ [x:y:z] \in
 \mP^2 : x^2+3xy+2y^2+yz=0 \}$. It is obvious that $C$ passes through
 $p_1$ and $p_5$ and  $D$ is a smooth conic in $\mP^2$ tangent to
 $L_2$ at $p_4$ which also passes through $p_2$, $p_3$, and $p_6$
 (Figure~\ref{fig-pencil-2}).

 We next construct an elliptic fibration on $E(1)$ from the cubic
 pencil $\lambda(A+B) + \mu(L_1 + L_2 + L_3)$ as follows: Blow up at
 $p_1$ and blow up at the intersection point of the proper transform
 of $B$ with the exceptional curve $e_1$. Let $e_1'$ be an
 exceptional divisor induced by the last blowing up. Blow up at
 $p_2$, $p_3$, $p_5$, and $p_6$; let $e_2$, $e_3$, $e_5$, and $e_6$
 be exceptional curves induced by the blowing ups, respectively.
 Finally, blow up at $p_4$ and blow up at the intersection point of
 the proper transform of $B$ with the exceptional curve $e_4$. And then
 blow up again at the intersection point of the proper transform of $B$
 with the exceptional curve $e_4'$ induced by the second blowing up.
 Let $e_4''$ be an exceptional divisor induced by the last blowing up.
 We get an elliptic fibration $E(1) = \mP^2 \# 9\overline{\mP}^2$
 over $\mP^1$. We denote this elliptic fibration by
 $g : Y=E(1) \to \mP^1$ (Figure~\ref{fig-Y-2}).
 Note that there is an $I_6$-singular fibration on $g : Y \to \mP^1$
 which consists of the proper transforms of $L_1$, $L_2$, $L_3$, $e_1$,
 $e_4$, and $e_4'$. There is also a $I_2$-singular fiber on $g : Y
 \to \mP^1$ which consists of the proper transforms of $A$ and $B$,
 denoted by $\widetilde{A}$ and $\widetilde{B}$, respectively.
 Furthermore, there is another $I_2$-singular fiber on $g : Y \to
 \mP^1$ which consists of the proper transforms of $C$ and $D$,
 denoted by $\widetilde{C}$ and $\widetilde{D}$. According to the
 list of Persson~\cite{Pers}, there exist only two more nodal singular
 fibers. On the other hand, $e_1'$, $e_2$, $e_3$, $e_4''$, $e_5$, and
 $e_6$ in $Y$ are the sections of $g : Y \to \mP^1$. Among these
 sections, we use only $e_2$, $e_5$, and $e_6$ in the following
 constructions. The section $e_2$ connects the proper transform of
 $L_1$ and $\widetilde{B}$, $\widetilde{D}$. The section $e_5$
 connects the proper transform of $L_2$ and $\widetilde{A}$,
 $\widetilde{C}$. The section $e_6$ connects the proper transform of
 $L_3$ and $\widetilde{A}$, $\widetilde{D}$. We denote the three
 sections $e_2$, $e_5$, $e_6$ by $S_1$, $S_2$, $S_3$, respectively
 (Figure~\ref{fig-Y-2}).

\begin{figure}[hbtb]
 \begin{center}
 \setlength{\unitlength}{1mm}
 \includegraphics[height=4cm]{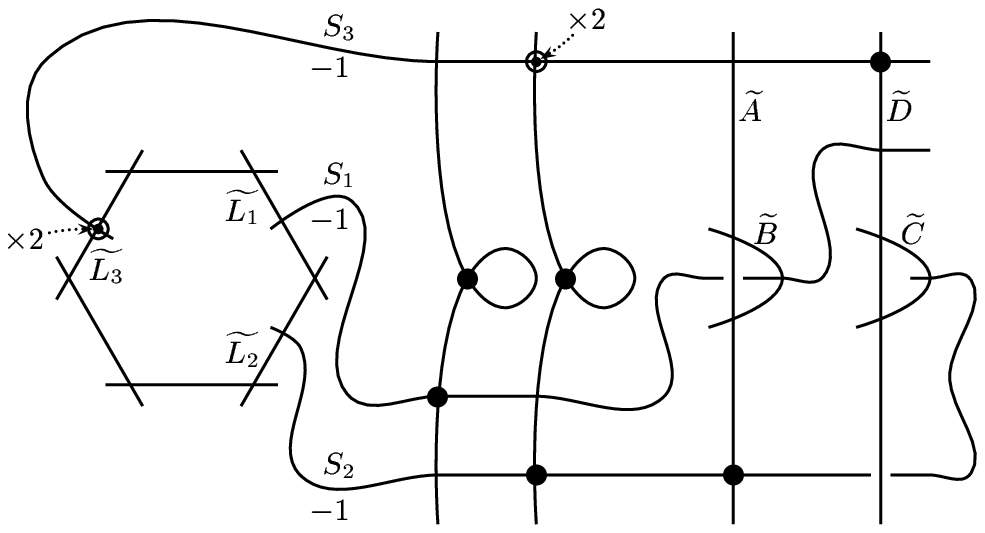}
 \end{center}
 \vspace*{-1 em}
 \caption{An elliptic fibration $Y=E(1)$}
 \label{fig-Y-2}
\end{figure}

 Let $Z':= Y \# 2\overline{\mP}^2$ be a rational surface obtained by
 blowing up at two singular points of two nodal fibers of $Y$.
 Let us denote the proper transforms of two nodal fibers by $F_1$ and
 $F_2$.
 And we blow up once at the marked points $\bullet$ on sections.
 We again blow up twice at the two marked points $\bigodot$ on the
 section $S_3$.
 Then we get a surface $Z := Y \# 10 \overline{\mP}^2$,
 which contains two disjoint linear chains of $\mP^1$:
 $C_{48, 17}=
 \uc{-3}-\uc{-6}-\uc{-5}-\uc{-3}-\uc{-2}-\uc{-2}-\uc{-2}-\uc{-3}-\uc{-2}$
 (which contains the proper transforms of $\widetilde{A}$, $S_3$,
 $F_1$, $S_2$, a part of the $I_6$-singular fiber), $C_{7,4}=
 \uc{-2}-\uc{-6}-\uc{-2}-\uc{-3}$ (which contains the proper
 transforms of $F_2$, $S_1$, $\widetilde{D}$) (Figure~\ref{fig-Z-2}).

\begin{figure}[hbtb]
 \begin{center}
 \setlength{\unitlength}{1mm}
 \includegraphics[height=5cm]{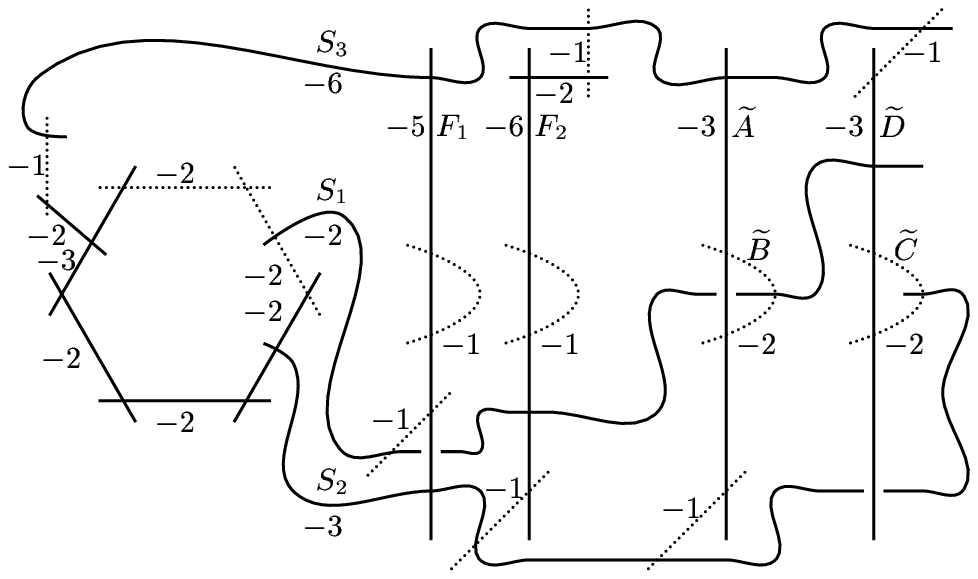}
 \end{center}
 \vspace*{-1 em}
 \caption{A rational surface $Z=Y\sharp 10{\overline \mP}^2$}
 \label{fig-Z-2}
\end{figure}

 Finally, by contracting these two disjoint chains of $\mP^1$ from $Z$,
 we obtain a projective surface $X$ with two singularities of class
 $T$. The existence of a $\mQ$-Gorenstein smoothing for $X$ is
 obtained by the same proof as in Section~4. Let us denote a general
 fiber of the $\mQ$-Gorenstein smoothing by $X_t$. Then, by the same
 argument in Section~5, we see that $X_t$ is a minimal surface of
 general type with $p_g=0$ and $K^3=3$. Furthermore, the rational
 blow-down $4$-manifold $Z_{48,7}$ of the rational surface $Z$ is simply
 connected, which can be proved in a similar way as in the proof of
 Theorem~\ref{thm-3.1}.
 Therefore we conclude that $X_t$ is a simply connected,
 minimal, complex surface of general type with $p_g=0$ and $K^2=3$.

\m

\n {\bf Remark.} One can find more surfaces using different
 configurations. For example, using an elliptic fibration on $E(1)$
 which has one $I_5$-singular fiber, one $I_3$-singular fiber, one
 $I_2$-singular fiber and two nodal fibers, we can construct another
 simply connected minimal surfaces of general type with $p_g=0$ and $K^2=3$.
 But we do not know whether all these constructions above provide
 the same deformation equivalent type of surfaces with $p_g=0$ and $K^2=3$.
 It is also a very interesting problem to determine whether these examples
 constructed above are diffeomorphic (or deformation equivalent)
 to the surface constructed in Section 3.

\b

\section{A simply connected symplectic $4$-manifold with $b_2^+=1$ and $K^2=4$}
\label{sec-7}

 In this section we construct a new simply connected symplectic
 $4$-manifold with $b_2^+=1$ and $K^2 =4$ using a rational blow-down
 surgery, and then we discuss the existence of a complex structure on
 it by using $\mQ$-Gorenstein smoothing theory.

 We first consider the elliptic fibration $g : Y=E(1) \to \mP^1$ used
 in the main construction in Section~3, which has one $I_8$-singular fiber,
 one $I_2$-singular fiber, and three sections $S_1$, $S_2$, and $S_3$.
 Among these sections, we use only $S_1$ and $S_2$
 in the following construction (Figure~\ref{fig-Y-3}).

\begin{figure}[hbtb]
 \begin{center}
 \setlength{\unitlength}{1mm}
 \includegraphics[height=4cm]{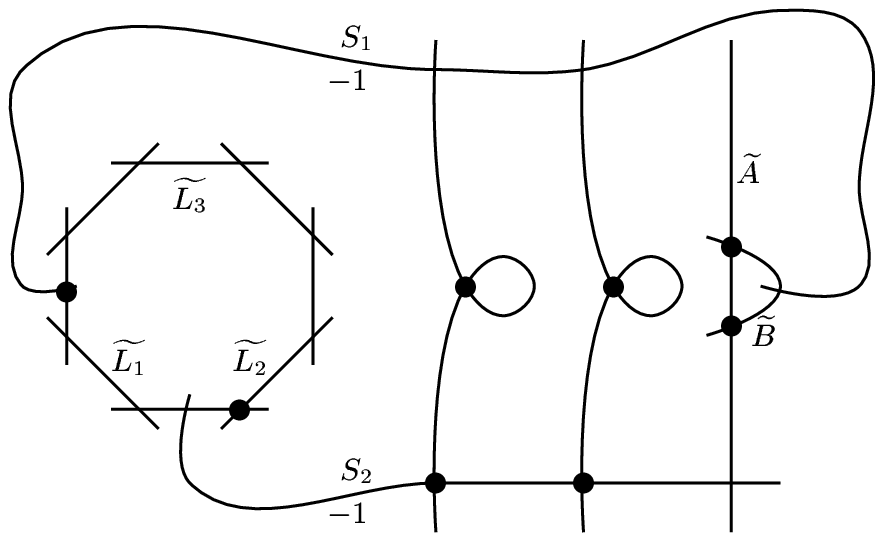}
 \end{center}
 \vspace*{-1 em}
 \caption{An elliptic fibration $Y=E(1)$}
 \label{fig-Y-3}
\end{figure}

 Let $Z':=Y \sharp 2\overline{\mP}^2$ be a rational surface obtained by
 blowing up at two nodal points of two nodal singular fibers of $Y$.
 We denote the proper transforms of two nodal fibers by $F_1$ and
 $F_2$. Blowing up once at the six marked points $\bullet$ on
 $S_2$, the $I_8$-singular fiber and the $I_2$-singular fiber,
 we get a rational surface $Z'':=Y\sharp 8\overline{\mP}^2$
 (Figure~\ref{fig-Zdp-3}).

 \begin{figure}[hbtb]
 \begin{center}
 \setlength{\unitlength}{1mm}
 \includegraphics[height=4.5cm]{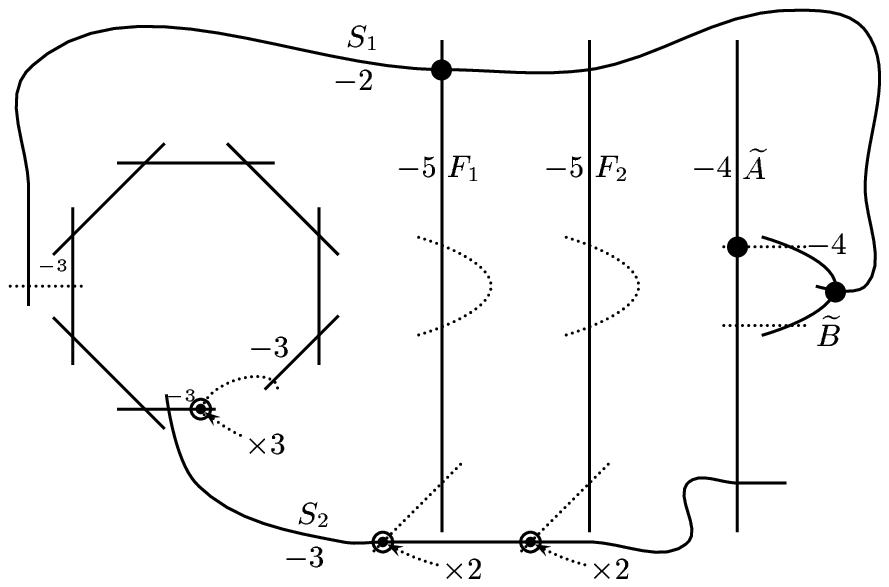}
 \end{center}
 \vspace*{-1 em}
 \caption{A rational surface $Z''=Y\sharp 8{\overline \mP}^2$}
 \label{fig-Zdp-3}
\end{figure}

 Finally,
 blow up once at the three marked points $\bullet$ on the proper
 transforms of $S_1$ and $\widetilde{A}$, and blow up again
 $3$-times, $2$-times, and $2$-times as indicated in Figure~\ref{fig-Zdp-3}
 at the three marked points $\bigodot$ on the proper transforms
 of $I_9$-singular fibers and $S_2$ respectively.
 Then we get a rational surface $Z:=Y \sharp 18\overline{\mP}^2$,
 which contains four disjoint linear chains of $\mP^1$:
 $C_{131,27}=\uc{-5}-\uc{-7}-\uc{-6}-\uc{-2}-\uc{-3}-\uc{-2}-\uc{-2}
  -\uc{-2}-\uc{-2}-\uc{-3}-\uc{-2}-\uc{-2}-\uc{-2}$
 (which contains the proper transforms of $\widetilde{A}$, $S_2$,
 and the $I_8$-singular fiber),
 $C_{7,2}=\uc{-4}-\uc{-5}-\uc{-2}-\uc{-2}$
 (which contains the proper transforms of $S_1$ and $F_2$),
 $C_{4,1} = \uc{-6}-\uc{-2}-\uc{-2}$
 (which contains the proper transform of $F_1$), and
 $C_{3,1} = \uc{-5}-\uc{-2}$
 (which contains the proper transform of $\widetilde{B}$)
 (Figure~\ref{fig-Z-3}).

\begin{figure}[hbtb]
 \begin{center}
 \setlength{\unitlength}{1mm}
 \includegraphics[height=5cm]{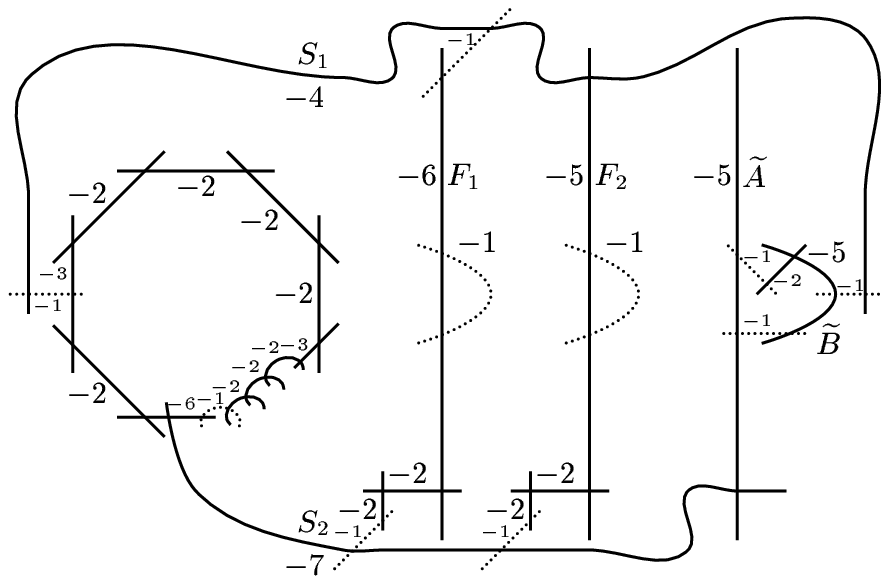}
 \end{center}
 \vspace*{-1 em}
 \caption{A rational surface $Z=Y\sharp 18{\overline \mP}^2$}
 \label{fig-Z-3}
\end{figure}

 Now we perform a rational blow-down surgery of the surface
 $Z= Y\sharp 18\overline{\mP}^2$ as before.
 Note that the surface $Z$ contains four disjoint configurations
 - $C_{131,27},\, C_{7,2},\, C_{4,1}$ and $C_{3,1}$.
 Let us decompose the surface $Z$ into
 \[Z= Z_{0}\cup\{C_{131,27}\cup C_{7,2} \cup C_{4,1} \cup C_{3,1}\}.\]
 Then the $4$-manifold, say  $Z_{131,7,4,3}$, obtained by
 rationally blowing down along the four configurations can be decomposed into
 \[Z_{131,7,4,3}=Z_{0}\cup\{B_{131,27}\cup B_{7,2}\cup B_{4,1}\cup B_{3,1}\},\]
 where $B_{131,27},\, B_{7,2},\, B_{4,1}$ and $B_{3,1}$
 are the corresponding rational balls.  Then we have

\begin{theorem}
\label{thm-7.1}
  The rational blow-down $Z_{131,7,4,3}$ of the surface $Z$
  in the construction above is a simply connected symplectic
  $4$-manifold with $b_2^+=1$ and $K^2 =4$ which is homeomorphic,
  but not diffeomorphic, to a rational surface $\mP^2\sharp 5\overline{\mP}^2$.
\end{theorem}

\begin{proof}
 Since all proofs except the last statement are parallel to those
 of Theorem~\ref{thm-3.1}, we omit it.
 Furthermore, one can easily show that the standard canonical class $K_{Z}$
 of $Z$ induces a non-trivial Seiberg-Witten basic class
 $\widetilde{K_{Z}}$ on $Z_{131,7,4,3}$ (refer to Theorem 3 in~\cite{P2}
 or Theorem 3.1 in~\cite{SS}).
 It means that $Z_{131,7,4,3}$ is not diffeomorphic to a rational surface
 $\mP^2\sharp 5\overline{\mP}^2$.
\end{proof}

\n {\bf Remarks}. 1. In fact, one can prove that the symplectic
 $4$-manifold $Z_{131,7,4,3}$ constructed in Theorem~\ref{thm-7.1}
 above is minimal by using a technique in~\cite{OS}. \\
 2. Recently, several authors constructed a simply connected minimal
 symplectic $4$-manifold with $b_2^+=1$ and $K^2 =4$ using a
 fiber-sum technique and a Luttinger surgery (\cite{Ak, BK}).
 We do not know whether the symplectic $4$-manifold $Z_{131,7,4,3}$
 constructed in Theorem~\ref{thm-7.1} above is diffeomorphic to one
 of their constructions. Despite of the fact,
 there is a big difference between our construction and theirs.
 That is, our example $Z_{131,7,4,3}$ has more room to investigate
 the existence of a complex structure, which is discussed below.

\m

 We close this article by discussing the possibility of the existence of a
 complex structure on the symplectic $4$-manifold $Z_{131,7,4,3}$.
 One way to approach this problem is to use $\mQ$-Gorenstein smoothing
 theory as in Section~4.
 That is, by contracting four disjoint chains of $\mP^1$ from $Z$ as before,
 we first obtain a projective surface $X$ with four singularities of class
 $T$. And then we investigate a $\mQ$-Gorenstein smoothing for $X$.
 It is known that the cohomology $H^2(T^0_X)$ contains the obstruction space
 of a $\mQ$-Gorenstein smoothing of $X$
 (refer to Proposition~\ref{pro-2.4} and Theorem~\ref{thm-2.1}).
 That is, if $H^2(T^0_X)=0$, then there is a $\mQ$-Gorenstein smoothing of $X$.
 But the cohomology $H^2(T^0_X)$ is not zero in our case,
 so that it is hard to determine
 whether there exists a $\mQ$-Gorenstein smoothing of $X$.
 Therefore we need to develop more $\mQ$-Gorenstein smoothing theory in
 order to investigate the existence of a complex structure
 on the symplectic $4$-manifold $Z_{131,7,4,3}$.
 We leave this question for future research.

\m

\n {\bf Open Problem.} Determine whether the symplectic $4$-manifold
 $Z_{131,7,4,3}$ admits a complex structure.

\b
\b


\begin{thebibliography}{999}

\bibitem[1]{Ak}   A. Akhmedov, {\it Small exotic $4$-manifolds}, math.GT/0612130

\bibitem[2]{Artin} M. Artin, {\it Some numerical criteria for contractability of
                   curves on algebraic surfaces}, Amer. J. Math. {\textbf 84}
                   (1962), 485--496.

\bibitem[3]{BK} S. Baldridge and P. Kirk, {\it Constructions of small symplectic
                $4$-manifolds using Luttinger surgery}, math.GT/0703065

\bibitem[4]{B}   R. Barlow, {\it A simply connected surface of general type with
                $p_g =0$}, Invent. Math. {\textbf 79} (1984), 293--301.

\bibitem[5]{BHPV} W. Barth, K. Hulek, C. Peters, A. Van de Ven,
                  {\it Compact complex surfaces}, 2nd ed. Springer-Verlag,
                  Berlin, 2004.

\bibitem[6]{FS} R. Fintushel and R. Stern, {\it Rational blowdowns of smooth
                 4-manifolds}, Jour. Diff. Geom. {\textbf 46} (1997), 181--235.

\bibitem[7]{FZ} H. Flenner and M. Zaidenberg, {\it $\mQ$-acyclic
                 surfaces and their deformations}, Contemp. Math.
                 {\textbf 162} (1994), 143--208.

\bibitem[8]{KSB} J. Koll\'ar and Shepherd-Barron, {\it Threefolds and
                  deformations of surface singularities}, Invent. Math.
                  {\textbf 91} (1988), 299--338.

\bibitem[9]{LP}  Y. Lee and J. Park, {\it A simply connected surface of general
                  type with $p_g=0$ and $K^2=2$}, to appear in
                  Invent. Math. 2007

\bibitem[10]{LS}  S. Lichtenbaum and M. Schlessinger, {\it The cotangent complex
                  of a morphism}, Trans. Amer. Math. Soc. {\textbf 128}
                  (1967), 41--70.

\bibitem[11]{LW}  E. Looijenga and J. Wahl, {\it Quadratic functions and smoothing
                  surface singularities}, Topology {\textbf 25} (1986), 261--291.

\bibitem[12]{Man91} M. Manetti, {\it Normal degenerations of the complex projective
                    plane}, J. Reine Angew. Math. {\textbf 419} (1991), 89--118.

\bibitem[13]{Man01} M. Manetti, {\it On the moduli space of diffeomorphic algebraic
                    surfaces}, Invent. Math. {\textbf 143} (2001), 29--76.

\bibitem[14]{Mu}  D. Mumford, {\it Topology of normal singularities and a criterion
                  for simplicity}, Publ. Math. IHES {\textbf 36} (1961), 229--246.

\bibitem[15]{Nak} N. Nakayama, {\it Zariski-decomposition and abundance},
                  MSJ Memoirs, {\textbf 14}, Math. Soc.of Japan, 2004.

\bibitem[16]{OS}  P. Ozsv\'{a}th and Z. Szab\'{o}, {\it On Park's exotic
                  smooth four-manifolds}, Geometry and Topology of
                  Manifolds, Fields Institute Communications {\textbf 47}
                  (2005), 253--260.

\bibitem[17]{Pal} V. P. Palamodov, {\it Deformations of complex spaces},
                  Russian Math. Surveys {\textbf 31:3} (1976), 129--197.

\bibitem[18]{P1}  J. Park, {\it Seiberg-Witten invariants of generalized rational
                  blow-downs}, Bull. Austral. Math. Soc.
                  {\textbf 56} (1997), 363--384.

\bibitem[19]{P2}   J. Park, {\it Simply connected symplectic 4-manifolds with
                  $b_2^+=1$ and $c_1^2=2$}, Invent. Math. {\textbf 159} (2005),
                  657--667.

\bibitem[20]{Pers} U. Persson, {\it Configuration of Kodaira fibers on rational
                   elliptic surfaces}, Math. Z. {\textbf 205} (1990), 1--47.

\bibitem[21]{SS}  A. Stipsicz and Z. Szab\'o, {\it An exotic smooth structure on
                   $\mC\mP\sp 2\#6\overline{\mC\mP\sp 2}$}, Geometry and
                   Topology  {\textbf 9} (2005), 813--832.

\bibitem[22]{Sy1}   M. Symington, {\it Symplectic rational blowdowns}, Jour.
                    Diff. Geom. {\textbf 50} (1998), 505--518.

\bibitem[23]{Sy2}   M. Symington, {\it Generalized symplectic rational blowdowns},
                    Algebraic and Geometric Topology {\textbf 1} (2001), 503--518.

\bibitem[24]{Wa1}  J. Wahl, {\it Smoothing of normal surface singularities},
                  Topology {\textbf 20} (1981), 219--246.

\bibitem[25]{Wa2} J. Wahl, {\it Elliptic deformations of minimally elliptic
                  singularities}, Math. Ann. {\textbf 253} (1980),
                  no. 3, 241--262.

\end{thebibliography}
\end{document}